\documentclass{article}
\usepackage[margin=1.5in]{geometry}
\usepackage{graphicx} 
\usepackage{amsthm}
\usepackage{amssymb}
\usepackage{amsmath} 
\usepackage{mathtools}
\usepackage{xcolor}
\usepackage{enumerate}
\usepackage{microtype}
\usepackage[backend=bibtex,natbib,style=numeric,sorting=nyt]{biblatex}
\usepackage{hyperref}
\newcommand{\dens}{d}

\newtheorem{thm}{Theorem}

\newtheorem{lemma}[thm]{Lemma}

\newtheorem{cor}[thm]{Corollary}
\newtheorem{problem}{Problem}

\def\inst#1{$^{#1}$}

\title{A General Upper Bound on Multicolor \\Ordered Ramsey Numbers\thanks{M. Balko was supported by grant no. 23-04949X of the Czech Science Foundation (GA\v{C}R) and by the Center for Foundations of Modern Computer Science (Charles Univ. project UNCE 24/SCI/008).
}}

\author{
Martin Balko\inst{1}
\and
Kl\'{a}ra Grinerov\'{a}\inst{1}
}

\begin{document}

\maketitle

\begin{center}

{\footnotesize
\inst{1} 
Department of Applied Mathematics, \\
Faculty of Mathematics and Physics, Charles University, Czech Republic \\
\texttt{\string{balko,klarka\string}@kam.mff.cuni.cz}
}

\end{center}

\begin{abstract}
We provide a general upper bound on multicolor ordered Ramsey numbers in terms of the interval chromatic number and the degeneracy of an ordered graph. 
We extend previous results by Conlon, Fox, Lee, and Sudakov (2017) by showing that for every $n$-vertex ordered graph $G^<$ with degeneracy $d\geq2$, and interval chromatic number $\chi$, its $q$-color ordered Ramsey number satisfies $r_<(G^<;q) \in n^{O(d^{q-1}\lceil \log{\chi}\rceil^{q-1})}$ for every $q \geq 2$. 
For fixed parameters $q,d,\chi$, the resulting estimate is polynomial in $n$. 
For triangle-free ordered graphs $G^<$, we also provide the stronger estimate $n^{O(q^2d {\lceil \log \chi \rceil}^{q-1})}$.
It also follows from a recent result by Li (2026) that our upper bound is almost tight for ordered matchings.
\end{abstract}

\section{Introduction}

Ramsey theory is an active area of research, based on the idea that every sufficiently large structure contains a well-organized substructure. 
A cornerstone result is the \emph{Ramsey theorem} stating that, for every integer $q \geq 2$ and every graph $G$, there exists a positive integer $N$ such that every $q$-coloring of the edges of the complete graph $K_N$ with colors from $[q]=\{1,\dots,q\}$ contains a monochromatic copy of $G$. 
The smallest such $N$ is called the \emph{multicolor Ramsey number} of $G$ and is denoted by $r(G;q)$.

The growth rate of Ramsey numbers is extensively studied.
The classical bound $r(K_n;2)\leq4^n$ on Ramsey numbers of complete graphs~\cite{erdosszekeres1935} was improved over the years only by subexponential factors~\cite{conlon2009upper,thomason1988}. 
A breakthrough was achieved by Campos, Griffiths, Morris, and Sahasrabudhe~\cite{campos2023exponential}, who proved that $r(K_n;2)\leq(4-\varepsilon)^n$ for some constant $\varepsilon>0$ and all sufficiently large $n$, giving the first exponential improvement of the upper bound. Gupta, Ndiaye, Norin, and Wei~\cite{gnnw24} subsequently improved this to
\[
r(K_n;2)\leq(4e^{-0.14/e})^{n+o(n)},
\]
where $4e^{-0.14/e}\approx3.7992$. 

Since every graph $G$ on $n$ vertices is a subgraph of $K_n$, we have $r(G;2)\leq r(K_n;2)$. 
This bound is of the correct exponential order for sufficiently dense graphs. 
However, for sparse graphs, the situation is very different. Chv\'{a}tal, R\"{o}dl, Szemer\'{e}di, and Trotter~\cite{chvatal1983bounded} showed that graphs with bounded maximum degree have Ramsey numbers that grow only linearly in the number of vertices. 
More generally, the Burr--Erd\H{o}s conjecture~\cite{burr1975}, verified by Lee~\cite{lee17}, states that for every $d$ there exists a constant $c=c(d)$ such that every $d$-degenerate graph $G$ on $n$ vertices satisfies $r(G;2)\leq cn$. 

In the multicolor case, Ramsey numbers are even less understood, but there has been substantial progress recently. 
For $q\geq3$, the classical general upper bound is $r(K_n;q)<q^{qn}$. 
Balister, Bollob\'{a}s, Campos, Griffiths, Hurley, Morris, Sahasrabudhe, and Tiba~\cite{balister2024} obtained an exponential improvement for every fixed $q\geq2$.
These were further improved~\cite{yangmao2026} and currently the best bound~\cite{jo2026} gives
\[
r(K_n;q)\leq \exp\left(-\frac{cn}{q(\log(2q))^2}\right)q^{qn}
\]
for $q\geq2$ and $n\geq Cq^2(\log(2q))^6$, with suitable absolute constants $c,C>0$.

There has also been very exciting progress on multicolor lower bounds. 
A breakthrough by Conlon and Ferber~\cite{conlon2021multicolorlower} was followed by several further improvements~\cite{avm26,camPo6,sawin2022lower,wigderson2021lower}, until, recently, Brada\v{c}~\cite{bradac2026} proved that for every fixed $q\geq3$,
\[
r(K_n;q)\geq c_q2^{(q-1)n/2}
\]
for some constant $c_q>0$ and all sufficiently large $n$.
For multicolor Ramsey numbers of general graphs, we note that the result of Lee~\cite{lee17} about Ramsey numbers of $d$-degenerate graphs extends to an arbitrary number of colors.

Motivated by these recent breakthroughs, we study the multicolor \emph{ordered Ramsey numbers}, which are a natural variant where the vertices are equipped with a linear order. While ordered Ramsey numbers have been studied quite extensively in the two-color case, there are only sporadic results for more colors.

In this paper, we obtain general upper bounds in the multicolor setting.
In particular, we obtain polynomial estimates for ordered graphs with bounded degeneracy and bounded interval chromatic number, extending two-color bounds by Conlon, Fox, Lee, and Sudakov~\cite{conlon2017ordered}.
This gives a first general nontrivial upper bound on multicolor ordered Ramsey numbers.

\section{Preliminaries}

All logarithms are to base two. We omit ceiling and floor signs whenever they are not crucial.
To avoid trivial cases, we consider only graphs with at least one edge in our estimates.

An \emph{ordered graph} $G^<$ is a graph with a linear order on its vertex set. 
Let $H^<$ and $G^<$ be ordered graphs. We say that $H^<$ is an \emph{ordered subgraph} of $G^<$ if there is an increasing injection $f\colon V(H)\to V(G)$ such that $f(u)f(v)\in E(G)$ whenever $uv\in E(H)$.

The \emph{interval chromatic number} of an ordered graph $G^<$ is the minimum number of intervals into which the vertex set of $G^<$ can be partitioned so that no two vertices in the same interval are adjacent. Further, for a nonnegative integer $d$, a graph $G$ is called \emph{$d$-degenerate} if every nonempty subgraph of $G$ contains a vertex of degree at most $d$. The \emph{degeneracy} of $G$ is the smallest $d$ such that $G$ is $d$-degenerate. A \emph{degeneracy ordering} $v_1,\dots,v_n$ is an ordering in which each vertex has at most $d$ earlier neighbors; it need not agree with the given vertex order.
The \emph{density} between two nonempty disjoint sets $A$ and $B$ of vertices is the number of edges of $G$ between $A$ and $B$ divided by $|A||B|$.

We write $K_m^<(\chi)$ for the complete $\chi$-partite ordered graph with $m$ vertices in each part, where the parts are consecutive intervals. Thus, $K_m^<(2)=K_{m,m}^<$. For vertex sets, $W_1<\cdots<W_r$ means that every vertex of $W_i$ precedes every vertex of $W_j$ whenever $i<j$. Such sets need not be intervals in the original vertex set, and we call them \emph{ordered sets}.

Let $q \ge 2$ be an integer and let $G^<_1,\dots,G^<_q$ be ordered graphs. The $q$-color \emph{ordered Ramsey number} $r_<(G^<_1,\dots,G^<_q)$ is the smallest integer $N$ such that every coloring of the edges of $K^<_N$ with $q$ colors contains a monochromatic ordered copy of $G^<_i$ in color $i$ for some $i \in [q]$. If all graphs are the same, we write $r_<(G^<;q)$. In the two-color case, we write $r_<(G^<)$.
It is easy to see that multicolor ordered Ramsey numbers are between the classical Ramsey numbers~\cite{balko2020ordered}. 
In particular, all multicolor ordered Ramsey numbers are finite and thus well-defined.

\section{Related Work}

Ordered Ramsey numbers for two colors have been studied a lot in recent years; see, for example,~\cite{balko2019,balko2020ordered,balko2019bounded,balko2024,conlon2017ordered,gishboliner2024,rohatgi2019}. For more than two colors, there are only a few results, mostly for special types of ordered graphs; see, for example,~\cite{balko2020ordered,falgas23,geneson2019,girao24,moshkovitz2012,mubayi17}. 
In this paper, we study the multicolor case for general ordered graphs with given degeneracy and interval chromatic number.

In the two-color setting, Conlon, Fox, Lee, and Sudakov~\cite{conlon2017ordered} proved a general upper bound for ordered graphs in terms of their interval chromatic number and degeneracy. In particular, if $G^<$ is an ordered $d$-degenerate graph on $n$ vertices with interval chromatic number $\chi$, then
\begin{equation}
\label{eq_2}
r_<(G^<;2) \leq n^{32d\log \chi}.
\end{equation}
That implies that ordered graphs with bounded degeneracy and bounded interval chromatic number have two-color ordered Ramsey numbers at most polynomial in $n$. In particular, for some sparse graphs, the ordered Ramsey number is much smaller than the general exponential bound that follows from classical Ramsey theory. 
On the other hand, ordered Ramsey numbers of sparse ordered graphs can grow much faster than ordinary Ramsey numbers, as there are ordered $n$-vertex 1-regular graphs $M^<$ with $r_<(M^<;2) \in n^{\Omega(\log{n}/\log{\log{n}})}$~\cite{balko2020ordered,conlon2017ordered}.

In the multicolor setting, much less is known. 
For example, Conlon, Fox, Lee, and Sudakov~\cite{conlon2017ordered} proved that
\begin{equation}
\label{eq_3}
r_<(M^<;q) \leq n^{(2 \log n)^{q-1}}
\end{equation}
for every $n$-vertex ordered graph $M^<$ of maximum degree one.
To complement this bound, very recently, Li~\cite{li26} showed that  for every integer $q \geq 2$, there is a constant $c_q > 0$ such that almost every 1-regular ordered graph $M^<$ on $n$ vertices, as even $n$ tends to infinity, satisfies
\begin{equation}
\label{eq-li}
r_<(M^<;q) \geq n^{c_q(\log{n})^{q-1}/(\log{\log{n}})^{q-1}},
\end{equation}
negatively answering a problem by Conlon, Fox, Lee, and Sudakov~\cite{conlon2017ordered} for $q \geq 3$.
In our previous work~\cite{clanek}, for every $q\geq2$ and every 1-regular ordered graph on $n$ vertices with interval chromatic number two, we proved the stronger bound $r_<(M^<;q)\leq4(n/2)^q$.

More generally, we also proved that every ordered $d$-degenerate graph $G^<$ on $n$ vertices with interval chromatic number two satisfies
\begin{equation}
\label{eq_1}
r_<(G^<; q) \leq n^{\frac{(d + 2)(d + 1)^{q - 1} - 2}{d}} \in n^{O((d+1)^{q-1})}.
\end{equation}
Concerning lower bounds, we proved~\cite{clanek} that for every integer $q\geq 2$, there exist $c_q > 0$ such that for every $d$ and every sufficiently large integer $n$ there is an ordered $d$-regular graph $G^<$ on $2n$ vertices with interval chromatic number two satisfying
\begin{equation}
\label{eq-lowerBound}
r_<(G^<;q) \geq c_q \frac{(dn)^{q/4}}{(\log dn)^q}.
\end{equation}

Unfortunately, in this case, there is a huge gap between the lower bound and the upper bound. 
Also, these results apply only to ordered graphs with interval chromatic number two or maximum degree one. 
Our new results make progress in this direction by giving upper bounds for ordered graphs with arbitrary interval chromatic number.

\section{Our Results}

We first give an improved upper bound on multicolor ordered Ramsey numbers of ordered graphs with bounded degeneracy and interval chromatic number two.

\begin{thm}
\label{thm-bipartite}
For all integers $d \geq 1$ and $q \geq 2$, every ordered $d$-degenerate graph $H^<$ on $n$ vertices with interval chromatic number two and maximum degree $\Delta$ satisfies
\[r_<(H^<;q) \leq (\Delta^{3/2} n)^{(1+o(1))dq^2 } ,\]
where $o(1)\to0$ as $n\to\infty$, uniformly in the other parameters.
\end{thm}

In fact, we prove a stronger statement as we give an estimate on the off-diagonal multicolor ordered Ramsey numbers where one of the ordered graphs is replaced by an ordered complete multipartite graph; see Theorem~\ref{thm_d_degenerated_bipartite}. For fixed $d$, the estimate from Theorem~\ref{thm-bipartite}, together with $\Delta\leq n$, improves the dependence on $q$ in the exponent from exponential to quadratic compared with~\eqref{eq_1} by Balko and Grinerová~\cite{clanek}. 

The techniques that we use also give an estimate for triangle-free ordered graphs with a larger interval chromatic number. The bound is similar to the one in Theorem~\ref{thm-bipartite}, with an additional factor depending on $\chi$ and $q$ in the exponent, but it gives a worse constant in the exponent.

\begin{thm}\label{thm_triangleFree}
For all integers $d \geq 1$ and $q \geq 2$, every ordered $d$-degenerate triangle-free graph $H^<$ on $n$ vertices with interval chromatic number $\chi$ satisfies
\[r_<(H^< ;q) \leq n^{21q^2d {\lceil \log \chi \rceil}^{q-1}}.\]
\end{thm}

Finally, we state our most general upper bound, which covers all ordered graphs. 

\begin{thm}
\label{thm_d_degenerated_multipartite_result}
For all integers $d\geq 2$ and $q\geq2$, let $H^<$ be an ordered  $d$-degenerate graph on $n\geq3$ vertices, with interval chromatic number $\chi$ and the degeneracy ordering $v_1,\dots,v_n$.
Let $p \geq 2$  be an integer such that every edge $v_iv_j$ of $H^<$ with $i<j$ has at most $p-1$ common neighbors among $v_1,\dots,v_{i-1}$. 
Then,
\[
r_<(H^<;q)  \leq n^{60dp^{q-2}\lceil\log\chi\rceil^{q-1}}.
\]
\end{thm}

Since for $d$-degenerate graphs the parameter $p$ can always be chosen to be at most $d$, we immediately obtain the following estimate.

\begin{cor}
\label{cor-general}
\label{cor_d_degenerated_multipartite_result}
For all integers $d\geq 2$ and $q\geq2$, every ordered $d$-degenerate graph $H^<$ on $n\geq3$ vertices, with interval chromatic number $\chi$ satisfies
\[
 r_<(H^<;q)\leq n^{60d^{q-1}\lceil\log\chi\rceil^{q-1}}.
\]
\end{cor}
Compared to the preceding results, there is a worse dependence on $d$ for $q\geq3$, but the resulting bound is still polynomial for fixed $d$, $q$, and $\chi$ and quasi-polynomial for fixed $d$ and $q$.
Theorem~\ref{thm_d_degenerated_multipartite_result} extends the earlier result~\eqref{eq_2} by Conlon, Fox, Lee, and Sudakov to the multicolor setting, up to an absolute constant in the exponent. 

It follows from the result~\eqref{eq-li} by Li~\cite{li26} that our upper bound is almost tight for ordered matchings. For $\Delta=d=1$ and any fixed $q$, using $\chi\leq n$, the upper bound in Theorem~\ref{thm_triangleFree} matches~\eqref{eq-li} up to a factor of $\Theta((\log\log n)^{q-1})$ in the exponent.

The proofs of all the main results are based on the same, rather complicated, embedding method. The more general the upper bound is, the more difficult some steps become, and some arguments require weaker intermediate estimates. To make the arguments easier to follow, we present the three proofs consecutively, starting with the simplest.

\section{Open Problems}

We believe that the term $d^{q-1}$ in the exponent of the estimate from Corollary~\ref{cor-general} can be improved, similarly to the estimates for triangle-free ordered graphs.
A substantial gap remains between the upper and lower bounds for fixed $d$ and $\chi$.
Apart from~\eqref{eq-lowerBound}, we are not aware of any nontrivial lower bounds in the multicolor case.

\begin{problem}
For fixed $d,\chi\geq2$, improve the lower bounds for multicolor ordered Ramsey numbers of $d$-degenerate graphs with interval chromatic number $\chi$.
\end{problem}

For $d\geq2$, the bound in Corollary~\ref{cor-general} beats the classical estimate $q^{qn}$ if $n^{60d^{q-1}\lceil\log\chi\rceil^{q-1}}$ is asymptotically smaller than $q^{qn}$. Thus, it would be interesting to improve the available bounds in the remaining parameter ranges. For $q=2$, Conlon, Fox, Lee, and Sudakov~\cite[Theorem~1.5]{conlon2017ordered} already proved $r_<(H^<)\leq2^{Cd\log^2(2n/d)}$ for an absolute constant $C>0$, which is particularly relevant when $d$ is large.

\begin{problem}
For $q \geq 3$, improve the upper bounds on $r_<(H^< ;q)$ for $n$-vertex ordered $d$-degenerate graphs $H^<$ with $d\geq2$ and interval chromatic number $\chi$ that satisfy
\[
n^{60d^{q-1}\lceil\log\chi\rceil^{q-1}}\geq q^{qn}.
\]
\end{problem}

\paragraph{Declaration on AI use.}

The original arguments were developed without AI. AI tools were subsequently used to polish the writing, to review the proofs, and to assist in improving the dependence on $d$.

\section{Proof of Theorem~\ref{thm-bipartite}}

Here, we prove Theorem~\ref{thm-bipartite}. In fact, we prove a stronger statement as we give an estimate on the off-diagonal multicolor ordered Ramsey numbers where one of the ordered graphs is replaced by an ordered complete multipartite graph; see Theorem~\ref{thm_d_degenerated_bipartite}.

To obtain our estimate, we use the following two results by Conlon, Fox, Lee, and Sudakov~\cite{conlon2017ordered}. Both results are stated in terms of edges and nonedges, which is equivalent to the two-color setting. The first lemma shows that if an ordered graph avoids an ordered subgraph with a fixed degeneracy, then its vertex set contains many large sets $W_1<\cdots<W_{2^s}$ with only a few edges between them.

\begin{lemma}[\cite{conlon2017ordered}] \label{lemma_conlon_sets}
Let $H^<$ be an ordered $d$-degenerate graph on $n$ vertices with maximum degree~$\Delta$. Assume that we have $c \in (0,1)$ and a positive integer $s$ such that $N \geq \big(2\Delta n (2^{s} c^{-1})^{d}\big)^{s}$. If an ordered graph on vertex set $[N]$ does not contain an ordered copy of $H^<$, then there exist sets $W_{1},W_{2},\dots,W_{2^{s}}\subset[N]$ such that
\begin{itemize}
\item[(i)]
for all $i$, $|W_i| \ge\frac{c^{sd} N}{(2^{sd + 1} \Delta n)^{s}}$,
\item[(ii)]
for $i<j$, all vertices in $W_{i}$ precede all vertices in $W_{j}$,
\item[(iii)]
for $i<j$, the density of edges between $W_{i}$ and $W_{j}$ is at most $c$.
\end{itemize}
\end{lemma}

The second lemma states that once we have ordered vertex sets with sufficiently large edge density between them, we can embed a large ordered complete multipartite graph.

\begin{lemma}[\cite{conlon2017ordered}] \label{lem_conlon_final}
Let $K^<=K_{n}^<(\chi)$ be a trivially ordered complete $\chi$-partite graph. If an ordered graph $G^<$ on the vertex set $[N]$ is such that there exist sets $W_{1},W_{2},\dots,W_{\chi}$ satisfying the following three conditions:
\begin{itemize}
\item[(i)]
for all $i$, $|W_{i}|\ge4\chi n$,
\item[(ii)]
for $i< j$, the vertices in $W_{i}$ precede the vertices in $W_{j}$,
\item[(iii)]
for $i<j$, the density of non-edges between $W_{i}$ and $W_{j}$ is at most $\frac{1}{8\chi^2 n}$,
\end{itemize}
then $G^<$ contains a copy of $K^<$.
\end{lemma}

The next result is a simple observation that allows us to control the density of edges between ordered sets. We turn a density condition between two large sets into a density condition between two families of smaller ordered subsets. Starting from two large vertex sets with few blue edges between them, we partition suitable subsets into ordered blocks and retain blocks with controlled densities.

\begin{lemma}\label{lemma_divide_into_intervals_sparse}
Let $n,m$ be positive integers, $c\geq0$, and $W_1<W_2$ be two vertex sets, each containing at least $N=2nm$ vertices, that induce a complete bipartite ordered graph colored blue and red. Assume that the density of blue edges between $W_1$ and $W_2$ is at most $c$. Then there exist sets
\[
U_1<\cdots<U_n\subseteq W_1,\qquad U_{n+1}<\cdots<U_{2n}\subseteq W_2,
\]
each of size $m$, such that the density of blue edges between $U_i$ and $U_j$ is at most $2cn$ whenever $1 \leq i\leq n<j \leq 2n$.
\end{lemma}

\begin{proof}
First, we choose subsets of size $N$ in the two parts with blue density at most $c$. Such a choice exists by averaging over independent uniformly chosen subsets of size $N$. We partition each chosen subset, in its inherited order, into $2n$ blocks of size $m$. We call a pair of blocks from different parts \emph{bad} if its blue density exceeds $2cn$. If $c=0$, there are no bad pairs. Otherwise, there are fewer than $2n$ bad pairs, since $2n$ such pairs would contain more than $2n\cdot 2cn\cdot m^2=cN^2$ blue edges.

Consider the bipartite graph whose edges are the bad pairs. Its edges can be covered by at most $n$ block-vertices in each part, as if there are at most $n$ nonisolated vertices in the first part, take all of them; otherwise, take any $n$ nonisolated vertices there. In the latter case, at least $n$ edges are covered, so at most $n$ edges remain, and their endpoints in the second part cover them. Remove the blocks in this cover and retain $n$ of the remaining blocks from each part, in order. These blocks have the required properties.
\end{proof}

Now we prove the main technical lemma, which extends the idea from the proof of Lemma~\ref{lemma_conlon_sets} to a setting with three colors. Roughly speaking, this lemma states that if an ordered graph colored red-blue-green with a low density of blue edges does not contain a red copy of the given ordered graph $H^<$, then there are two large sets with few non-green edges between them. The proof follows a greedy embedding of the vertices of $H^<$ in red according to a degeneracy ordering of $H^<$. At each step, we maintain candidate sets for the remaining vertices and either successfully embed the next vertex or obtain a pair of large sets with controlled density of blue and red edges.

\begin{lemma}\label{lemma_one_step_density_coloring}
For positive integers $n,m,d,\Delta$ with $d\leq\Delta$, and real numbers $c \in (0,1)$ and $\beta\geq c$, let $H^<$ be an ordered $d$-degenerate graph on $n$ vertices with interval chromatic number two and maximum degree $\Delta$. Let an ordered complete bipartite graph, with its first part preceding its second, have its two parts partitioned into ordered sets $U_1<\cdots<U_n$ and $U_{n+1}<\cdots<U_{2n}$, each of size at least $L\geq2m\Delta/c^d$. Color its edges red, blue, and green. If the blue density between each pair $U_i,U_j$ with $i\leq n<j$ is at most $\beta$, and there is no red ordered copy of $H^<$, then there are sets $W_1<W_2$ in different parts, each of size at least $m$, whose non-green density is at most $32n\Delta^3\beta$.
\end{lemma}

\begin{proof}
Let $v_1,\dots,v_n$ be a degeneracy ordering of $H$. If $v_i$ is the $j$th vertex in the vertex ordering of $H^<$ within the first interval of $H^<$, then we let $V_i=U_j$; if it is the $j$th vertex in the second interval, then we let $V_i=U_{n+j}$. This assignment ensures that choosing one vertex from each $V_i$ preserves the vertex order of $H^<$.

We use $d_{\mathrm B}(X,Y)$ to denote the density of blue edges between $X$ and $Y$. We write $N_R(w)$ for the red neighborhood of $w$. At step $t$, we maintain vertices $w_i\in V_i$ for each $i\leq t$ and candidate sets $V_{i,t}\subseteq V_i$ for every $i>t$ satisfying the following four conditions:
\begin{itemize}
\item[(i)] If $i,j\leq t$ and $v_iv_j\in E(H)$, then $w_iw_j$ is red.
\item[(ii)] If $j\leq t<i$ and $v_iv_j\in E(H)$, then every edge from $w_j$ to $V_{i,t}$ is red.
\item[(iii)] For $i>t$, $|V_{i,t}|\geq c^{d_{i,t}}|V_i|$, where $d_{i,t}$ is the number of neighbors of $v_i$ among $v_1,\dots,v_t$.
\item[(iv)] Finally,
\[P_t = \sum_{\substack{t<i<j\\v_iv_j\in E(H)}}
(1+1/d)^{-(d_{i,t}+d_{j,t})}d_{\mathrm B}(V_{i,t},V_{j,t})\leq \beta|E(H)|.
\]
\end{itemize}

Initially, we set $V_{i,0}=V_i$ and $d_{i,0}=0$, so that $P_0\leq \beta|E(H)|$, and all the four conditions are satisfied. 
Note that if the process reaches step $n$, then we have a red copy of $H^<$.

For an integer $t \geq 1$, assume that step $t-1$ has been completed. We let $I_t=\{i>t:v_tv_i\in E(H)\}$ and observe that $|I_t|\leq\Delta$, as $H^<$ has maximum degree $\Delta$. Since $H$ is bipartite, the set $I_t$ is independent. Thus, if $i\in I_t$ and $j>t$ is adjacent to $i$, then $j\notin I_t$ and its candidate set remains $V_{j,t-1}$. For $i\in I_t$ and each nonempty $S\subseteq V_{i,t-1}$,  we define
\[
\Phi_i(S)=\sum_{\substack{j>t\\v_iv_j\in E(H)}}
(1+1/d)^{-(d_{i,t-1}+d_{j,t-1})}d_{\mathrm B}(S,V_{j,t-1}).
\]
All weights $(1+1/d)^{-(d_{i,t-1}+d_{j,t-1})}$ and comparison sets $V_{j,t-1}$ in this definition remain fixed during the step.

For $w\in V_{t,t-1}$ and $i\in I_t$, call a set $S_i$ \emph{$w$-suitable} if it lies in the red neighborhood of~$w$ in $V_{i,t-1}$, has size at least $c|V_{i,t-1}|$, and satisfies $\Phi_i(S_i)\le
(1+1/d)\Phi_i(V_{i,t-1})$.
If some $w\in V_{t,t-1}$ has a $w$-suitable set for every $i\in I_t$, take $w_t=w$, set $V_{i,t}=S_i$ for $i\in I_t$, and leave all other candidate sets unchanged. Properties (i)--(iii) then hold. Every surviving edge incident with $I_t$ has exactly one endpoint in $I_t$, so its weight decreases by a factor of $1+1/d$. The new total contribution of these edges is
\[
(1+1/d)^{-1}\sum_{i\in I_t}\Phi_i(S_i)
\leq\sum_{i\in I_t}\Phi_i(V_{i,t-1}),
\]
which is their old total contribution. Other surviving terms are unchanged, and terms incident with $v_t$ disappear. Thus, we get $P_t\leq P_{t-1}$, proving (iv).

Otherwise, for each $w\in V_{t,t-1}$, some index $i\in I_t$ has no $w$-suitable set. By the pigeonhole principle, a fixed index $i$ fails for a set $W\subseteq V_{t,t-1}$ with $|W|\geq|V_{t,t-1}|/\Delta$. We fix such an index $i$ and set $A=V_{i,t-1}$ and $R=A$. We keep all sets $V_{j,t-1}$ fixed during the following deletion process.

While some $w\in W$ has at least $c|A|$ red neighbors in $R$, delete its entire red neighborhood $S$ in $R$. 
Since $S$ is not $w$-suitable, $\Phi_i(S)>(1+1/d)\Phi_i(A)$. 
The deleted sets are disjoint, and $|S|\Phi_i(S)$ is additive in $S$. 
Hence, if $D$ is their union, then $|D| = \sum_{S\text{ deleted}}|S|$ and
\[
(1+1/d)|D|\Phi_i(A) \leq \sum_{S\text{ deleted}}|S|\Phi_i(S)\leq|A|\Phi_i(A).
\]
If $\Phi_i(A)=0$, no such deletion is possible, as then any deleted non-empty subset $S$ of $A$ has $\Phi_i(S)=0$ and would be $w$-suitable. 
Otherwise, the total number of deleted vertices is $|D| \leq |A|/(1+1/d)=d|A|/(d+1)$. 
Hence, when the process stops, $|R| = |A| - |D| \geq|A|/(d+1)$ and every $w\in W$ has fewer than $c|A|$ red neighbors in $R$. The red density between $W$ and $R$ is therefore at most $(d+1)c$.

The term corresponding to $v_tv_i$ in (iv) at step $t-1$, together with $|W|\geq|V_{t,t-1}|/\Delta$ and $|R|\geq|V_{i,t-1}|/(d+1)$, bounds the blue density between $W$ and $R$ by
\[
\Delta(d+1)|E(H)|(1+1/d)^{d_{t,t-1}+d_{i,t-1}}\beta.
\]
Here $d_{t,t-1}\leq d$ and $d_{i,t-1}\leq d-1$, since $v_t$ is one more earlier neighbor of $v_i$. Since $c\leq\beta$, the non-green density is at most
\begin{align*}
\bigl(\Delta(d+1)|E(H)|(1+1/d)^{2d-1}+d+1\bigr)\beta
&=\bigl(\Delta d|E(H)|(1+1/d)^{2d}+d+1\bigr)\beta\\
&\leq (e^2/2+2)n\Delta^3\beta\\
&\leq32n\Delta^3\beta.
\end{align*}
In the first inequality, we used $|E(H)|\leq n\Delta/2$, $d\leq\Delta$, and $(1+1/d)^{2d}\leq e^2$.

Moreover, $ |W|\geq\frac{c^dL}{\Delta}\geq m$, and $|R|\geq\frac{c^dL}{d+1}\geq m$, since $d+1\leq2\Delta$. The sets $W$ and $R$ lie in different parts, and labeling them $W_1<W_2$ completes the proof.
\end{proof}

We are now ready to prove the main technical result of this section, which provides an upper bound on multicolor ordered Ramsey numbers when all but one of the graphs are $d$-degenerate ordered graphs with interval chromatic number two. The proof is based on an iterative argument in which we gradually control the density of edges in already used colors. In each step, we either embed the required ordered graph in the corresponding color or obtain a pair of large sets with controlled edge densities, which allows us to proceed to the next color. In the final step, this structure forces a monochromatic copy of the last ordered bipartite graph.

\begin{thm}\label{thm_d_degenerated_bipartite}
For all integers $m\geq1$ and $q\geq2$, let $H_1^<,\dots,H_{q-1}^<$ be ordered $d$-degenerate graphs on $n$ vertices with interval chromatic number two and maximum degree $\Delta$. Then
\[
r_<(H_1^<,\dots,H_{q-1}^<,K_{m,m}^<)\leq N,
\]
where
\begin{align*}
N={}&2^{d+1}\Delta n\,(32m)^{d(q-1)}\,8m
\left(\frac{4\Delta n}{(2n)^d}\right)^{q-2}
\bigl(64\Delta^3n^2\bigr)^{d(q-2)(q+1)/2}.
\end{align*}
\end{thm}

\begin{proof}
We set $c_1=(32m(64\Delta^3n^2)^{q-2})^{-1}$, and $c_i=c_1(64\Delta^3n^2)^{i-1}$ for $1\leq i\leq q-1$. We also define $N_1=\frac{c_1^dN}{2^{d+1}\Delta n}$ and
\[
N_i=\frac{N_{i-1}(2nc_{i-1})^d}{4\Delta n}
\]
for every $i$ with $2\leq i\leq q-1$. The choice of $N$ is equivalent to
\[
N=8m\,\frac{2^{d+1}\Delta n}{c_1^d}
\prod_{i=2}^{q-1}\frac{4\Delta n}{(2nc_{i-1})^d}.
\]
Thus, $N_{q-1}=8m$ and $c_{q-1}=1/(32m)$.

The backward recurrence shows that all $N_i$ are integers, and that $N_{i-1}$ is divisible by $2n$ for $2\leq i\leq q-1$. Indeed, $c_{i-1}=1/(32m(64\Delta^3n^2)^{q-i})$, so each factor $4\Delta n/(2nc_{i-1})^d$ is an integer divisible by $2n$.

Consider a $q$-coloring of $K_N^<$. If it contains a copy of $H_1^<$ in color $1$, then we are done. Otherwise, we apply Lemma~\ref{lemma_conlon_sets} with $s=1$ and $c=c_1$ to the graph formed by edges of color  1 on the entire vertex set. 
We obtain sets $W_{1,1}<W_{2,1}$, each of size at least $N_1$, with density of edges of color 1 between them at most $c_1$.

Assume that, after step $i-1$, we have sets $W_{1,i-1}<W_{2,i-1}$, each of size at least $N_{i-1}$, with density at most $c_{i-1}$ in the union of colors $1,\dots,i-1$. Lemma~\ref{lemma_divide_into_intervals_sparse} then gives $n$ ordered subsets in each part, each of size $L = N_{i-1}/(2n)$, with density at most $2nc_{i-1}$ in these colors between every pair from different parts. Recolor colors $1,\dots,i-1$ blue, color $i$ red, and the remaining colors green. Since $L=2N_i\Delta/(2nc_{i-1})^d$ and $2nc_{i-1}\leq1/(1024m\Delta^3n)<1$, we can apply Lemma~\ref{lemma_one_step_density_coloring} with $c=\beta=2nc_{i-1}$ and $m=N_i$. We then either find a copy of $H_i^<$ in color $i$, or obtain sets $W_{1,i}<W_{2,i}$, each of size at least $N_i$, with non-green density at most
\[
32n\Delta^3\cdot2nc_{i-1}=c_i.
\]

If no ordered copy of $H_i^<$ is found in color $i$ for any $i\in[q-1]$, then the final sets have size at least $8m$ and density in the union of colors $1,\dots,q-1$ at most $1/(32m)$. Lemma~\ref{lem_conlon_final}, with $\chi=2$ and $n=m$, then gives a copy of~$K_{m,m}^<$ in color $q$.
\end{proof}

Theorem~\ref{thm-bipartite} now follows immediately from Theorem~\ref{thm_d_degenerated_bipartite}.

\begin{proof}[Proof of Theorem~\ref{thm-bipartite}]
The ordered graph $H^<$ in the statement is an ordered subgraph of~$K_{n,n}^<$. Thus if we apply Theorem~\ref{thm_d_degenerated_bipartite} with $m=n$ and all $H_i^<=H^<$, then
\begin{align*}
N&=2^{3dq^2+dq-8d+2q}\Delta^{q-1+3d(q-2)(q+1)/2}n^{d(q^2-q-1)+q}\\
& \leq 2^{4dq^2}\Delta^{q-1+3d(q-2)(q+1)/2}n^{dq^2}\\
&\leq 2^{4dq^2}(\Delta^{3/2}n)^{dq^2},
\end{align*}
where we used $q-1+3d(q-2)(q+1)/2\leq3dq^2/2$ and $d(q^2-q-1)+q\leq dq^2$ for $d\geq1$ and $q\geq2$. Consequently,
\[
r_<(H^<;q)\leq(16\Delta^{3/2}n)^{dq^2}
\leq(\Delta^{3/2}n)^{(1+4/\log n)dq^2},
\]
which proves the desired asymptotic bound.
\end{proof}

\section{Proof of Theorem~\ref{thm_triangleFree}}

In this section, we extend Theorem~\ref{thm-bipartite} to triangle-free ordered graphs with arbitrarily large interval chromatic numbers. The proof techniques are quite similar to those in the previous section, though some steps are more complicated. We again prove a stronger statement as we give an estimate on the off-diagonal multicolor ordered Ramsey numbers where one of the ordered graphs is replaced by an ordered complete multipartite graph; see Theorem~\ref{thm_d_degenerated_multipartite}.

We first state and prove a multipartite analog of Lemma~\ref{lemma_divide_into_intervals_sparse}, which works in the non-bipartite setting, but gives worse bounds on the density.

\begin{lemma}
\label{lemma_divide_into_intervals_multipartite}
For positive integers $n,m$ and $r\geq2$, let $U_1<\cdots<U_r$ be vertex sets, each containing at least $rnm$ vertices, forming an ordered complete $r$-partite graph colored blue and red. Assume that the blue density between $U_i$ and $U_j$ is at most $c\geq0$ for each $i<j$. Then there are sets $U_{i,1}<\cdots<U_{i,n}\subseteq U_i$ for each $i\in[r]$, each of size at least $m$, such that the blue density between $U_{i,k}$ and $U_{j,l}$ is at most $2r^2nc$ whenever $i<j$.
\end{lemma}
\begin{proof}
We partition each $U_i$, in the given vertex order, into $rn$ intervals with sizes differing by at most one. 
Each interval has size at least $m$ and at least $|U_i|/(2rn)$. 
We call a pair of intervals from different parts \emph{bad} if its blue density exceeds $2r^2nc$. If $c=0$, there are no bad pairs. 
Otherwise, for each pair of parts $U_i,U_j$, there are fewer than $2n$ bad pairs, as otherwise $2n$ bad pairs would contain more than
\[
2n\cdot2r^2nc\cdot\frac{|U_i|}{2rn}\frac{|U_j|}{2rn}
=c|U_i||U_j|
\]
blue edges, which is impossible.

As in the proof of Lemma~\ref{lemma_divide_into_intervals_sparse}, the bad pairs between $U_i$ and $U_j$ can be covered by at most $n$ intervals in each part. Take the union of these covers over all pairs of parts. At most $(r-1)n$ intervals are removed from any one part, so at least $n$ remain, and we retain $n$ of them in each part.
\end{proof}

The following lemma is a multipartite analog of Lemma~\ref{lemma_one_step_density_coloring} for triangle-free graphs with arbitrary interval chromatic number.

\begin{lemma}\label{lemma_one_step_density_coloring_non_bipartite}
For positive integers $n,m,d,\Delta$, an integer $\chi\geq2$, and $c\in(0,1)$, let $H^<$ be a triangle-free ordered $d$-degenerate graph on $n$ vertices with maximum degree $\Delta$ and interval chromatic number $\chi$. Let $U_1<\cdots<U_\chi$ be the parts of an ordered complete $\chi$-partite graph, each of size at least $4m\chi n\Delta/c^d$, with edges colored red, blue, and green. If the blue density between each two parts is at most $c$, and there is no red ordered copy of $H^<$, then there are sets $W_1<W_2$ in different parts, each of size at least $m$, with non-green density between them at most $64\Delta^3\chi^2n^2c$.
\end{lemma}

\begin{proof}
We apply Lemma~\ref{lemma_divide_into_intervals_multipartite} with $r=\chi$ and block-size parameter $\lceil2m\Delta/c^d\rceil$, temporarily identifying red and green. 
The part-size condition holds since $\chi n\lceil2m\Delta/c^d\rceil\leq4m\chi n\Delta/c^d$. 
We obtain sets $U_{j,1}<\cdots<U_{j,n}$ in each part, each of size at least $L=2m\Delta/c^d$, with blue density at most $\beta=2\chi^2nc$ between sets from different parts.

Let $v_1,\dots,v_n$ be a degeneracy ordering of $H$. If $v_i$ is the $k$th vertex in the original vertex order contained in the $j$th interval of an interval coloring of $H^<$, then we let $V_i=U_{j,k}$. This assignment preserves the given vertex order.

We now follow the proof of Lemma~\ref{lemma_one_step_density_coloring} with these candidate sets, shrinkage parameter $c$, and initial blue-density bound $\beta$, retaining properties (i)--(iv). We again write $d_{\mathrm B}(X,Y)$ for the blue density between $X$ and $Y$. In particular, we maintain
\[P_t=\sum_{\substack{t<i<j\\v_iv_j\in E(H)}}
(1+1/d)^{-(d_{i,t}+d_{j,t})}d_{\mathrm B}(V_{i,t},V_{j,t})\leq \beta|E(H)|.
\]
Initially, we set $V_{i,0}=V_i$ and $d_{i,0} = 0$, so that $P_0\leq \beta|E(H)|$ and the conditions are satisfied. 
If the process reaches step $n$, it gives a red ordered copy of $H^<$.

Assume that step $t-1$ has been completed. We let $I_t=\{i>t:v_tv_i\in E(H)\}$, and observe $|I_t|\leq\Delta$. Since $H$ is triangle-free, the set $I_t$ is independent. Thus, if $i\in I_t$ and $j>t$ is adjacent to $i$, then $j\notin I_t$ and its candidate set remains $V_{j,t-1}$. For $i\in I_t$ and each nonempty $S\subseteq V_{i,t-1}$, we set
\[
\Phi_i(S)=\sum_{\substack{j>t\\v_iv_j\in E(H)}}
(1+1/d)^{-(d_{i,t-1}+d_{j,t-1})}d_{\mathrm B}(S,V_{j,t-1}).
\]
All weights $(1+1/d)^{-(d_{i,t-1}+d_{j,t-1})}$ and comparison sets $V_{j,t-1}$ in this definition remain fixed during the step.

For $w\in V_{t,t-1}$ and $i\in I_t$, call a set $S_i$ \emph{$w$-suitable} if it lies in the red neighborhood of $w$ in $V_{i,t-1}$, has size at least $c|V_{i,t-1}|$, and satisfies $\Phi_i(S_i)\leq (1+1/d)\Phi_i(V_{i,t-1})$.

If some $w$ has a $w$-suitable set for every $i\in I_t$, take $w_t=w$, set $V_{i,t}=S_i$ for every $i\in I_t$, and leave all other candidate sets unchanged. 
Properties (i)--(iii) then hold. 
Every surviving edge incident with $I_t$ has exactly one endpoint in $I_t$, so its weight decreases by a factor of $1+1/d$. 
The new total contribution of these edges is
\[
(1+1/d)^{-1}\sum_{i\in I_t}\Phi_i(S_i)\leq\sum_{i\in I_t}\Phi_i(V_{i,t-1}),
\]
which is their old total contribution. Other surviving terms are unchanged, and terms incident with $v_t$ disappear. Thus $P_t\leq P_{t-1}$ and property~(iv) is satisfied.

Otherwise, for each $w\in V_{t,t-1}$, some index $i\in I_t$ has no $w$-suitable set. By the pigeonhole principle, a fixed index $i$ fails for a set $W\subseteq V_{t,t-1}$ with $|W|\geq|V_{t,t-1}|/\Delta$. Fix this $i$ and put $A=V_{i,t-1}$ and $R=A$. Keep all sets $V_{j,t-1}$ fixed during the following deletion process.

While some $w\in W$ has at least $c|A|$ red neighbors in $R$, delete its entire red neighborhood $S$ in $R$. Since $S$ is not $w$-suitable, $\Phi_i(S)>(1+1/d)\Phi_i(A)$. The deleted sets are disjoint, and $|S|\Phi_i(S)$ is additive in $S$. Hence, if $D$ denotes their union, then we have $|D| = \sum_{S\text{ deleted}}|S|$ and
\[
(1+1/d)|D|\Phi_i(A) \leq \sum_{S\text{ deleted}}|S|\Phi_i(S)\leq|A|\Phi_i(A).
\]
If $\Phi_i(A)=0$, no such deletion is possible, as then any deleted non-empty subset $S$ of $A$ has $\Phi_i(S)=0$ and would be $w$-suitable.  
Otherwise, the total number of deleted vertices is $|D| \leq |A|/(1+1/d)=d|A|/(d+1)$. Hence, when the process stops, $|R| = |A| - |D| \geq|A|/(d+1)$ and every $w\in W$ has fewer than $c|A|$ red neighbors in $R$. The red density between $W$ and $R$ is therefore at most $(d+1)c$.

The term corresponding to $v_tv_i$ in $P_{t-1}$, together with $|W|\geq|V_{t,t-1}|/\Delta$ and $|R|\geq|V_{i,t-1}|/(d+1)$, bounds the blue density between $W$ and $R$ by at most
\[
\Delta(d+1)|E(H)|(1+1/d)^{d_{t,t-1}+d_{i,t-1}}\beta.
\]
Here $d_{t,t-1}\leq d$ and $d_{i,t-1}\leq d-1$, since $v_t$ is one more earlier neighbor of $v_i$. Since $c\leq\beta$, the non-green density is at most
\begin{align*}
\bigl(\Delta(d+1)|E(H)|(1+1/d)^{2d-1}+d+1\bigr)\beta
&=\bigl(\Delta d|E(H)|(1+1/d)^{2d}+d+1\bigr)\beta\\
&\leq(e^2/2+2)n\Delta^3\beta\\
&\leq32n\Delta^3\beta.
\end{align*}
In the first inequality, we used $|E(H)|\leq n\Delta/2$, $d\leq\Delta$, and $(1+1/d)^{2d}\leq e^2$.

Moreover, $|W|\geq\frac{c^dL}{\Delta}\geq m$, and $|R|\geq\frac{c^dL}{d+1}\geq m$, since $d+1\leq2\Delta$. All relevant edges join distinct interval classes, so $W$ and $R$ lie in different parts. We label them $W_1<W_2$, and note that their non-green density is at most $32n\Delta^3\beta=64\Delta^3\chi^2n^2c$, which finishes the proof.
\end{proof}

Now we prove the main technical step in the proof of Theorem~\ref{thm_triangleFree}. Roughly speaking, the lemma shows that in any $q$-coloring, we either find a monochromatic copy of one of the ordered graphs with given interval chromatic number and degeneracy, or we can find many large disjoint subsets of the vertex set so that the density of edges in all colors except the last one is low between every pair of subsets. The proof proceeds by double induction on the number of colors and the number of parts.

\begin{lemma}\label{lemma_multipartite}
For integers $n\geq8$, $m\geq1$, $q\geq2$, $\chi\geq2$, and $1\leq s\leq z=\lceil\log\chi\rceil$, let $H_1^<,\dots,H_{q-1}^<$ be ordered $d$-degenerate graphs on $n$ vertices with maximum degree $\Delta$ and interval chromatic number $\chi$. Assume that $H_i$ is triangle-free for $2\leq i\leq q-1$. For $c\in(0,1)$, let
\[
N_{q,s,m,c}=(64\Delta^3n^2\chi^2)^{q^2sdz^{q-2}}c^{-(q-1)sdz^{q-2}}2^{2(q-1)s^2dz^{q-2}}m.
\]
If an integer $N\geq N_{q,s,m,c}$ and a $q$-coloring of $K_N^<$ contains no ordered copy of $H_i^<$ in color $i$ for any $i\in[q-1]$, then there are sets $W_1<\cdots<W_{2^s}$, each of size at least $m$, such that the density of edges in colors $[q-1]$ between any two of them is at most $c$.
\end{lemma}

\begin{proof}
We use a double induction on $q$ and $s$. Notice that $64\Delta^3n^2\chi^2\geq2^{2z}$.
For $q=2$, the required conclusion follows from Lemma~\ref{lemma_conlon_sets} applied to the graph formed by edges of color 1. Indeed, we have
\[N_{2,s,m,c}=(64\Delta^3n^2\chi^2)^{4sd}c^{-sd}2^{2s^2d}m \geq c^{-sd}(2^{sd+1}\Delta n)^s m,\]
and thus, Lemma~\ref{lemma_conlon_sets} gives sets of size at least $m$.

Let $q\geq3$ and suppose first that $s=1$. We set $D=dz^{q-2}$, $c'=c/(64\Delta^3n^2\chi^2)$, and
\[
m'=\left\lceil\frac{4\chi^2n\Delta}{(c')^d}m\right\rceil
\leq 64\Delta^3n^2\chi^2(c')^{-d}m.
\]
Since $D\geq d$ and $64\Delta^3n^2\chi^2 \geq2^{2z}$, we obtain
\begin{align*}
N_{q-1,z,m',c'}
&=(64\Delta^3n^2\chi^2)^{(q-1)^2D}(c')^{-(q-2)D}2^{2(q-2)zD}m'\\
&\leq (64\Delta^3n^2\chi^2)^{((q-1)^2+q-2)D+d+1}
c^{-(q-2)D-d}2^{2(q-2)zD}m\\
&\leq (64\Delta^3n^2\chi^2)^{(q^2-3)D+d+1}c^{-(q-1)D}m\\
&\leq (64\Delta^3n^2\chi^2)^{q^2D}c^{-(q-1)D}m
\leq N_{q,1,m,c}.
\end{align*}

We identify colors $q-1$ and $q$ and apply the induction hypothesis for $q-1$ colors and $s=z$, with parameters $m',c'$. This gives $2^z$ ordered sets, and we retain the first $\chi$ of them. Each has size at least $m'$, and the density in colors $[q-2]$ between any two is at most $c'$. We replace colors $1,\dots,q-2$ with blue, color $q-1$ with red, and color $q$ with green. We then apply Lemma~\ref{lemma_one_step_density_coloring_non_bipartite} to these $\chi$ parts and $H_{q-1}^<$, which is triangle-free. The part-size condition follows from $m'\geq4m\chi n\Delta/(c')^d$. A red copy of $H^<_{q-1}$ is excluded by assumption, so the lemma gives two ordered sets of size at least $m$ with non-green density at most $64\Delta^3n^2\chi^2c'=c$. This proves the case $s=1$.

Now assume that $s\geq2$. We set $D=dz^{q-2}$ and $M=2\lceil N_{q,s-1,m,c}\rceil$, so that $M\leq4N_{q,s-1,m,c}$. Direct substitution gives
\[
N_{q,1,M,c/2^{s+1}}
\leq4\cdot 2^{-(q-1)D(3s-5)}N_{q,s,m,c}
\leq N_{q,s,m,c}.
\]
For the last inequality, we used $q\geq3$, $s\geq2$, and $z\geq s$, so that $(q-1)D(3s-5)\geq4$.

We apply induction for $q$ colors and $s=1$, with parameters $M$ and $c/2^{s+1}$, obtaining sets $U_1<U_2$, each of size at least $M$, with density at most $c/2^{s+1}$ in colors $[q-1]$. We remove from $U_1$ every vertex with more than $(c/2^s)|U_2|$ neighbors in $U_2$ in these colors. At most half the vertices are removed, and  the remaining set $U_1'$ has size at least $N_{q,s-1,m,c}$, so the induction hypothesis gives sets
\[
U_{1,1}<\cdots<U_{1,2^{s-1}}\subseteq U_1',
\]
each of size at least $m$, with pairwise density at most $c$ in colors $[q-1]$.

For each $j$, let $X_{2,j}$ consist of the vertices in $U_2$ with more than $c|U_{1,j}|$ neighbors in $U_{1,j}$ in colors $[q-1]$. The degree condition on $U_1'$ gives
\[
c|U_{1,j}|\,|X_{2,j}|
\leq\frac{c}{2^s}|U_{1,j}|\,|U_2|,
\]
so $|X_{2,j}|\leq|U_2|/2^s$. Hence, the set
\[
U_2'=U_2\setminus \left(\bigcup_{j=1}^{2^{s-1}}X_{2,j}\right)
\]
has size at least $|U_2|/2\geq N_{q,s-1,m,c}$. 
We apply induction inside $U_2'$ for $q$ colors and $s-1$ with parameters $m$ and $c$ to obtain sets $U_{2,1}<\cdots<U_{2,2^{s-1}}$ contained in $U_2'$, each of size at least $m$, such that the density of edges in colors $[q-1]$ between any two of these sets is at most $c$.
Every vertex of~$U_2'$ has at most $c|U_{1,j}|$ neighbors in $U_{1,j}$ in colors $[q-1]$, so the density in colors $[q-1]$ between any set from the first family and any set from the second family is at most $c$.
Taking the two families in this order gives the required $2^s$ sets.
\end{proof}

We now prove the main result of this section, which gives the final estimate as an immediate corollary. The theorem gives an upper bound on multicolor ordered Ramsey numbers by forcing either a monochromatic copy of one of the given ordered graphs with fixed interval chromatic number and degeneracy or a complete multipartite ordered graph in the remaining color.

\begin{thm}\label{thm_d_degenerated_multipartite}
For integers $m\geq1$, $\chi\geq2$, and $q\geq2$, let $H_1^<,\dots,H_{q-1}^<$ be ordered $d$-degenerate graphs on $n\geq8$ vertices with interval chromatic number $\chi$ and maximum degree~$\Delta$. Assume that $H_i$ is triangle-free for $2\leq i\leq q-1$, and put $L=\max\{n,m\}$. Then
\begin{align*}
r_<(H_1^<,\dots,H_{q-1}^<,K_m^<(\chi))
\leq(64\Delta^3L^2\chi^2)^{(q+1)^2d\lceil\log\chi\rceil^{q-1}}4\chi m.
\end{align*}
\end{thm}

\begin{proof}
To abbreviate, we set $z=\lceil\log\chi\rceil$, $c=1/(8\chi^2m)$, $A=64\Delta^3n^2\chi^2$, and $B=64\Delta^3L^2\chi^2$. Since $A\leq B$, $c^{-1}\leq B$, and $2^{2z}\leq B$, we obtain
\begin{align*}
N_{q,z,4\chi m,c}
&=A^{q^2dz^{q-1}}c^{-(q-1)dz^{q-1}}
2^{2(q-1)dz^q}4\chi m\\
&\leq B^{(q^2+2q-2)dz^{q-1}}4\chi m\\
&\leq B^{(q+1)^2dz^{q-1}}4\chi m.
\end{align*}

Consider an arbitrary $q$-coloring of the ordered complete graph on $B^{(q+1)^2dz^{q-1}}4\chi m$ vertices. 
If it contains a copy of some $H_i^<$ in color $i \in [q-1]$, then we are done. 
Otherwise, Lemma~\ref{lemma_multipartite} gives $2^z \geq \chi$ ordered sets, each of size at least $4\chi m$, with density between them at most $1/(8\chi^2m)$ in colors $[q-1]$. 
We retain the first $\chi$ sets, and apply Lemma~\ref{lem_conlon_final} to the ordered graph formed by edges of color $q$ with the parameter $n=m$, obtaining an ordered copy of $K_m^<(\chi)$ in color $q$.
\end{proof}

Finally, we deduce Theorem~\ref{thm_triangleFree} as an immediate corollary.

\begin{proof}[Proof of Theorem~\ref{thm_triangleFree}]
Put $z=\lceil\log\chi\rceil$. We assume $n\geq8$, since for $2\leq n<8$ the classical bound gives
\[
r_<(H^<;q)\leq r(K_n;q)\leq q^{qn}\leq n^{7q^2}\leq n^{21q^2dz^{q-1}}.
\]

Then, we apply Theorem~\ref{thm_d_degenerated_multipartite} with $m=n$ and $H_i^<=H^<$ for every $i$. Since $H^<$ is an ordered subgraph of $K_n^<(\chi)$, and $\Delta,\chi\leq n$, we obtain
\begin{align*}
r_<(H^<;q)
&\leq(64\Delta^3n^2\chi^2)^{(q+1)^2dz^{q-1}}4\chi n\\
&\leq n^{[9(q+1)^2+3]dz^{q-1}}\\
&\leq n^{21q^2dz^{q-1}}.
\end{align*}
Here we used $64n^7\leq n^9$ and $4n^2\leq n^3$ for $n\geq8$, and $9(q+1)^2+3\leq21q^2$ for $q\geq2$.
\end{proof}

\section{Proof of Theorem~\ref{thm_d_degenerated_multipartite_result}}

In this section, we extend the argument to general ordered graphs. The proof techniques are similar to those in the previous two cases. The main difference is that the two candidate sets corresponding to an edge may both be restricted in one step. We account for these simultaneous restrictions by an additional factor in the density invariant.

In this section, let $H^<$ be an ordered $d$-degenerate graph on $n\geq3$ vertices, with interval chromatic number $\chi$ and maximum degree $\Delta$. We assume that $d \geq 2$ and fix a degeneracy ordering $v_1,\dots,v_n$ of~$H^<$. 
We fix an integer $p$ with $2\leq p\leq d$ such that every edge $v_iv_j$ of $H^<$ with $i<j$ has at most $p-1$ common neighbors among $v_1,\dots,v_{i-1}$. The choice $p=d$ always satisfies this condition, since $v_j$ has at most $d-1$ earlier neighbors other than $v_i$. We aim to show that
\[
r_<(H^<;q)  \leq n^{60dp^{q-2}\lceil\log\chi\rceil^{q-1}}.
\]

For simplicity, we set $z=\lceil\log\chi\rceil$, and
\[
 A=\left\lceil\max\left\{2^{2z},2(d+1)(8\Delta^2\chi^2n^2)^{1/p}\right\}\right\rceil.
\]
We first prove a general version of Lemmas~\ref{lemma_one_step_density_coloring} and~\ref{lemma_one_step_density_coloring_non_bipartite}.

\begin{lemma}\label{lemma_one_step_density_coloring_general}
Let $H^<$, $p$, and the fixed degeneracy ordering be as above. 
For a positive integer $m$ and real numbers $\beta\geq0$ and $\lambda\in(0,1)$, let $U_1<\cdots<U_n$ be the parts of an ordered complete $n$-partite graph, each of size at least $L\geq2m\Delta\lambda^{-d}$, with edges colored red, blue, and green. 
If the blue density between each two parts is at most $\beta$, and there is no red ordered copy of $H^<$, then there are sets $W_1<W_2$ in parts corresponding to adjacent vertices of $H^<$, each of size at least $m$, with non-green density between them at most
\[
(d+1)\lambda+8\Delta(d+1)\beta\lambda^{-(p-1)}|E(H)|.
\]
Here $U_k$ corresponds to the $k$th vertex in the given vertex order of $H^<$.
\end{lemma}

\begin{proof}
Let $v_1,\dots,v_n$ be the fixed degeneracy ordering of $H$. If $v_i$ is the $k$th vertex in the given vertex order of $H^<$, then we let $V_i=U_k$. This assignment preserves the given vertex order. 
We again follow the embedding process from the proof of Lemma~\ref{lemma_one_step_density_coloring}. 
After embedding vertices $v_1,\ldots,v_t$, let
\[
d_{i,t}=|N_H(v_i)\cap\{v_1,\ldots,v_t\}|
\]
for every $i > t$ and
\[
b_{ij,t}=|N_H(v_i)\cap N_H(v_j)\cap\{v_1,\ldots,v_t\}|.
\]
for all distinct $i$ and $j$ with $t < i, j$.
At step $t$, we maintain vertices $w_i\in V_i$ for $i\leq t$ and candidate sets $V_{i,t}\subseteq V_i$ for $i>t$ satisfying the following four conditions:
\begin{itemize}
\item[(i)] If $i,j\leq t$ and $v_iv_j\in E(H)$, then $w_iw_j$ is red.
\item[(ii)] If $j\leq t<i$ and $v_iv_j\in E(H)$, then every edge from $w_j$ to $V_{i,t}$ is red.
\item[(iii)] For $i>t$, we have $|V_{i,t}|\geq\lambda^{d_{i,t}}|V_i|$.
\item[(iv)] We have
\[
P_t=\sum_{\substack{t<i<j\\v_iv_j\in E(H)}}
(1+1/d)^{-d_{i,t}-d_{j,t}}\lambda^{b_{ij,t}}
\dens_{\mathrm B}(V_{i,t},V_{j,t})
\leq \beta|E(H)|.
\]
\end{itemize}

Initially, set $V_{i,0}=V_i$. Conditions (i)--(iii) hold, and since $d_{i,0}=b_{ij,0}=0$ and the blue density between each pair of candidate sets is at most $\beta$, we have
\[
P_0=\sum_{\substack{i<j\\v_iv_j\in E(H)}}
\dens_{\mathrm B}(V_i,V_j)\leq \beta|E(H)|,
\]
proving (iv). If the process reaches step $n$, then (i) gives a red ordered copy of $H^<$.

Assume that step $t-1$ has been completed. 
We let $I_t=\{i>t:v_tv_i\in E(H)\}$, and observe that $|I_t|\leq\Delta$. 
Note that, this time, $I_t$ might not be an independent set. 
For $i\in I_t$ and $j>t$ with $v_iv_j\in E(H)$, we set $\kappa_{ij}=1$ if $j\notin I_t$ and $\kappa_{ij}=(2(1+1/d))^{-1}$ if $j\in I_t$. 
For each non-empty $S\subseteq V_{i,t-1}$, we also define
\[
\Phi_i(S)=\sum_{\substack{j>t\\v_iv_j\in E(H)}}
\kappa_{ij}(1+1/d)^{-d_{i,t-1}-d_{j,t-1}}
\lambda^{b_{ij,t-1}}\dens_{\mathrm B}(S,V_{j,t-1}).
\]
All weights $\kappa_{ij}(1+1/d)^{-d_{i,t-1}-d_{j,t-1}} \lambda^{b_{ij,t-1}}$ and comparison sets $V_{j,t-1}$ in this definition remain fixed during the step.

For a vertex $w\in V_{t,t-1}$, we call a non-empty set $S_i\subseteq V_{i,t-1}$ \emph{$w$-suitable} if all vertices of~$S_i$ are connected to $w$ by a red edge, $|S_i|\geq\lambda|V_{i,t-1}|$, and
\[
\Phi_i(S_i)\leq(1+1/d)\Phi_i(V_{i,t-1}).
\]

If some $w\in V_{t,t-1}$ admits a $w$-suitable $S_i$ for every $i\in I_t$, then we embed $v_t$ at $w_t=w$, set $V_{i,t}=S_i$ for all $i\in I_t$, and leave all other candidate sets unchanged. Properties (i) and (ii) follow from (ii) at step $t-1$ and the red-neighborhood condition on the sets $S_i$. For each $i\in I_t$, we have $d_{i,t}=d_{i,t-1}+1$, so
\[
|V_{i,t}|=|S_i|
\geq\lambda|V_{i,t-1}|
\geq\lambda^{d_{i,t-1}+1}|V_i|
=\lambda^{d_{i,t}}|V_i|.
\]
For all other unembedded vertices, both their candidate sets and the numbers of their embedded neighbors remain unchanged. Thus (iii) also holds.

To verify (iv), write
\[
x_{ij}=(1+1/d)^{-d_{i,t-1}-d_{j,t-1}}
\lambda^{b_{ij,t-1}}\dens_{\mathrm B}(S_i,V_{j,t-1})
\]
for an edge $v_iv_j$ with $i,j>t$ and an updated endvertex $v_i$. If only one endvertex of an edge is updated, say $v_i$, then $d_{i,t}=d_{i,t-1}+1$, while $d_{j,t}$ and $b_{ij,t}$ remain unchanged. Its new contribution is therefore $x_{ij}/(1+1/d)$.

If both endvertices are updated, then $v_t$ is adjacent to both $v_i$ and $v_j$, so $d_{i,t}=d_{i,t-1}+1$, $d_{j,t}=d_{j,t-1}+1$, and $b_{ij,t}=b_{ij,t-1}+1$. The new contribution of this edge is therefore
\[
(1+1/d)^{-d_{i,t-1}-d_{j,t-1}-2}
\lambda^{b_{ij,t-1}+1}\dens_{\mathrm B}(S_i,S_j).
\]
Since $S_j\subseteq V_{j,t-1}$, restricting the second set cannot increase the number of blue edges. Using $|S_j|\geq\lambda|V_{j,t-1}|$, we obtain
\[
\dens_{\mathrm B}(S_i,S_j)
\leq\frac{|V_{j,t-1}|}{|S_j|}
\dens_{\mathrm B}(S_i,V_{j,t-1})
\leq\lambda^{-1}\dens_{\mathrm B}(S_i,V_{j,t-1}).
\]
The factor $\lambda^{-1}$ cancels the additional factor $\lambda$ arising from the increment of $b_{ij,t-1}$. Hence the new contribution is at most $x_{ij}/(1+1/d)^2$. Interchanging $i$ and $j$ gives the bound $x_{ji}/(1+1/d)^2$ as well. Consequently, the new contribution is at most
\[
\frac{\min\{x_{ij},x_{ji}\}}{(1+1/d)^2}
\leq\frac{x_{ij}+x_{ji}}{2(1+1/d)^2}.
\]

Let $T_1$ and $T_2$ be the old total contributions of edges with exactly one and exactly two updated endvertices, respectively, excluding edges incident with $v_t$. More precisely, let
\[
T_1=
\sum_{\substack{t<i<j,\;v_iv_j\in E(H)\\
|\{i,j\}\cap I_t|=1}}
(1+1/d)^{-d_{i,t-1}-d_{j,t-1}}
\lambda^{b_{ij,t-1}}
\dens_{\mathrm B}(V_{i,t-1},V_{j,t-1})
\]
and
\[
T_2=
\sum_{\substack{t<i<j,\;v_iv_j\in E(H)\\
i,j\in I_t}}
(1+1/d)^{-d_{i,t-1}-d_{j,t-1}}
\lambda^{b_{ij,t-1}}
\dens_{\mathrm B}(V_{i,t-1},V_{j,t-1}).
\]
The coefficients $\kappa_{ij}$ allow us to bound the new total on these edges by
\[
(1+1/d)^{-1}\sum_{i\in I_t}\Phi_i(S_i).
\]
Indeed, an edge with exactly one updated endvertex appears once in this expression, contributing $x_{ij}/(1+1/d)$ when $i$ is its updated endvertex. An edge with both endvertices updated appears twice, and the choice $\kappa_{ij}=\kappa_{ji}=(2(1+1/d))^{-1}$ gives the combined contribution
\[
(1+1/d)^{-1}
\left(\frac{x_{ij}}{2(1+1/d)}
+\frac{x_{ji}}{2(1+1/d)}\right)
=\frac{x_{ij}+x_{ji}}{2(1+1/d)^2}.
\]
These are exactly the bounds obtained above.

By the suitability of the sets $S_i$, the new total on these edges is therefore at most
\[
(1+1/d)^{-1}\sum_{i\in I_t}\Phi_i(S_i)
\leq\sum_{i\in I_t}\Phi_i(V_{i,t-1})
=T_1+\frac{T_2}{1+1/d}
\leq T_1+T_2.
\]
Here the equality holds because each edge contributing to $T_1$ is counted once with coefficient $1$, whereas each edge contributing to $T_2$ is counted twice with coefficient $(2(1+1/d))^{-1}$. Other surviving contributions are unchanged, and the nonnegative contributions of edges incident with $v_t$ disappear. Thus $P_t\leq P_{t-1}\leq \beta|E(H)|$, proving (iv).

If the embedding fails at step $t$, every $w\in V_{t,t-1}$ has some $i\in I_t$ for which no $w$-suitable set exists. Since $|I_t|\leq\Delta$, the pigeonhole principle implies that we can fix such an $i$ and a set $W\subseteq V_{t,t-1}$ of size at least $|V_{t,t-1}|/\Delta$ for which the same obstruction holds for every $w\in W$.

Starting with $R=V_{i,t-1}$, whenever some $w\in W$ has at least $\lambda|V_{i,t-1}|$ red neighbors in $R$, we delete the set $S=R\cap N_R(w)$ from $R$. By the obstruction, $\Phi_i(S)>(1+1/d)\Phi_i(V_{i,t-1})$. The deleted sets are disjoint, and $|S|\Phi_i(S)$ is additive in $S$, since all weights and comparison sets remain fixed. Hence, if $D$ is their union, then $|D| = \sum_{S\text{ deleted}}|S|$ and
\[
(1+1/d)|D|\Phi_i(V_{i,t-1})
\leq \sum_{S\text{ deleted}}|S|\Phi_i(S)
\leq |V_{i,t-1}|\Phi_i(V_{i,t-1}).
\]
If $\Phi_i(V_{i,t-1})=0$, then no deletion is possible, as then any deleted non-empty subset $S$ of $V_{i,t-1}$ has $\Phi_i(S)=0$ and would be $w$-suitable.  
Otherwise, the last inequality and the strict inequalities for deleted sets imply $ |D|\leq\frac{|V_{i,t-1}|}{1+1/d} = \frac{d|V_{i,t-1}|}{d+1}. $ Thus, the remaining set satisfies $ |R| = |V_{i,t-1}| - |D|\geq\frac{|V_{i,t-1}|}{d+1}. $ When the process stops, every $w\in W$ has fewer than $\lambda|V_{i,t-1}|$ red neighbors in $R$. The density of red edges between $W$ and $R$ is therefore at most
\[
\frac{\lambda|V_{i,t-1}|}{|R|}
\leq(d+1)\lambda.
\]

The contribution of $v_tv_i$ to the potential in (iv) at step $t-1$, whose other terms are nonnegative, gives
\[
\dens_{\mathrm B}(V_{t,t-1},V_{i,t-1})
\leq |E(H)|\beta
(1+1/d)^{d_{t,t-1}+d_{i,t-1}}
\lambda^{-b_{ti,t-1}}.
\]
Here $d_{t,t-1}\leq d$, $d_{i,t-1}\leq d-1$, and $b_{ti,t-1}\leq p-1$. Since $(1+1/d)^{2d-1}<e^2<8$, the restriction to $W$ and $R$ yields
\[
\dens_{\mathrm B}(W,R)
\leq
\frac{|V_{t,t-1}||V_{i,t-1}|}{|W||R|}
\dens_{\mathrm B}(V_{t,t-1},V_{i,t-1})
\leq8\Delta(d+1)|E(H)|\beta\lambda^{-(p-1)}.
\]

Finally, by (iii), we have
\[
|W|\geq\lambda^dL/\Delta\geq m
\qquad\text{and}\qquad
|R|\geq\lambda^dL/(d+1)\geq m,
\]
using $d+1\leq2\Delta$. The sets $W$ and $R$ lie in parts corresponding to the adjacent vertices $v_t$ and $v_i$. Labeling them $W_1<W_2$ and adding the red and blue density bounds gives the required non-green density bound, which finishes the proof.
\end{proof}

\begin{lemma}\label{lemma_multipartite_step_general}
For a positive integer $m$ and $c\in(0,1)$, let $U_1<\cdots<U_\chi$ be the parts of an ordered complete $\chi$-partite graph, each of size at least $\tfrac12(A/c)^dm$, with edges colored red, blue, and green. If the blue density between each two parts is at most $(c/A)^p$, and there is no red ordered copy of $H^<$, then there are sets $W_1<W_2$ in different parts, each of size at least $m$, with non-green density at most $c$.
\end{lemma}

\begin{proof}
Set $\lambda=c/(2(d+1))$. Since $p\leq d$, the definition of $A$ gives
\[
A^d\geq(2(d+1))^d(8\Delta^2\chi^2n^2)^{d/p}
\geq8\Delta\chi n(2(d+1))^d.
\]
Consequently,
\[
\chi n\left\lceil2m\Delta\lambda^{-d}\right\rceil
\leq4\chi nm\Delta\lambda^{-d}
\leq\tfrac12(A/c)^dm.
\]
We apply Lemma~\ref{lemma_divide_into_intervals_multipartite} with $r=\chi$ and block-size parameter $\lceil2m\Delta\lambda^{-d}\rceil$, temporarily identifying red and green. 
We obtain $n$ ordered intervals in each part, each of size at least $2m\Delta\lambda^{-d}$, with blue density at most $\beta=2\chi^2n(c/A)^p$ between intervals from different parts.

Fix a partition of $V(H^<)$ into $\chi$ independent intervals, and assign distinct intervals in $U_j$, in order, to the vertices of its $j$th interval class. 
We add green edges between selected intervals from the same original part. 
Their blue density is zero, and no red ordered copy of $H^<$ is introduced. Lemma~\ref{lemma_one_step_density_coloring_general} therefore gives two sets of size at least $m$ whose non-green density is at most
\[
\begin{aligned}
(d+1)\lambda+8\Delta(d+1)|E(H)|\beta\lambda^{-(p-1)}
&\leq\frac c2+
\frac{4\Delta^2\chi^2n^2(2(d+1))^p}{A^p}c\\
&\leq c,
\end{aligned}
\]
where we used $|E(H)|\leq n\Delta/2$ and the definition of $A$. The two sets correspond to adjacent vertices of $H^<$, so they lie in different interval classes and hence in different original parts. Their colors are therefore unchanged, proving the lemma.
\end{proof}

We now prove the main technical step. We recursively define integers $b_2=d$ and
\begin{equation}\label{eq-recurrence}
b_q=pzb_{q-1}+d
\end{equation}
for every $q \geq 3$. Moreover, for integers $s \in [z]$ and $m\geq1$, and a real number $c \in (0,1)$, we set
\begin{equation}\label{eq-threshold}
 N_{q,s,m,c}=A^{3b_qs}c^{-b_qs}2^{2b_qs^2}m.
\end{equation}

\begin{lemma}\label{lemma_multipartite_general}
Any $q$-coloring of $K_N^<$ with $N\geq N_{q,s,m,c}$ and with no ordered copy of $H^<$ in colors $1,\ldots,q-1$ contains $2^s$ ordered sets of size at least $m$ whose pairwise densities in the union of these $q-1$ colors are at most $c$.
\end{lemma}
\begin{proof}
We use a double induction on $q$ and on $s$. For $q=2,s=1$, regard color $1$ as red and partition $[N]$ into $n$ consecutive sets, each of size at least $N/(2n)$, assigned to the vertices of $H^<$ in their given order. 
We then greedily embed in the degeneracy ordering $v_1,\ldots,v_n$, keeping candidate sets of size at least $c^{d_{i,t}}N/(2n)$. 
At step $t$, a choice of $w$ in the current candidate set for $v_t$ is \emph{admissible} if, for every unembedded neighbor $v_i$ of $v_t$, at least a proportion $c$ of its current candidate set is adjacent to $w$ by a red edge.

If no choice of $w$ is admissible, write $V_{i,t-1}$ for the current candidate set of each unembedded vertex $v_i$. Then, for some unembedded neighbor $v_i$ of $v_t$, every vertex of a set $W$ of size at least $|V_{t,t-1}|/\Delta$ in the current cell for $v_t$ has fewer than $c|V_{i,t-1}|$ red neighbors in $V_{i,t-1}$. 
The pair $W,V_{i,t-1}$ then has red density at most $c$ and both sets have size at least $c^dN/(2\Delta n)$. 
This is at least $m$, since $A^d\geq8\Delta\chi n(2(d+1))^d\geq2\Delta n$ and
\[
 N \geq N_{2,1,m,c}=A^{3d}c^{-d}2^{2d}m
 \geq 2\Delta n c^{-d}m.
\]
Since there is no red copy of $H^<$, there is a step with no admissible $w$. This proves the initial case.

Assume that $q\geq3$ and first let $s=1$. We set $c'=(c/A)^p$ and $ m'=\left\lceil\tfrac12(A/c)^d m\right\rceil  \leq(A/c)^d m$. We identify colors $q-1$ and $q$, leaving only $q-1$ colors, and apply the induction hypothesis with $s=z$, $m'$, and $c'$. The necessary size is available, since using $A\geq2^{2z}$ and $p+4\leq3p$, we have
\begin{align*}
 N_{q-1,z,m',c'}
 &\leq A^{(p+3)b_{q-1}z+d}c^{-pzb_{q-1}-d}2^{2b_{q-1}z^2}m\\
 &\leq A^{(p+4)b_{q-1}z+d}c^{-b_q}m\\
 &\leq A^{3b_q}c^{-b_q}m
 \leq N_{q,1,m,c}.
\end{align*}
We keep the first $\chi$ of the resulting $2^z$ ordered sets, and we color all edges in original colors $1,\ldots,q-2$ blue, those in color $q-1$ red, and those in color $q$ green. Since there is no red ordered copy of $H^<$, Lemma~\ref{lemma_multipartite_step_general} gives two sets of size at least $m$ with density at most $c$ in colors $1,\ldots,q-1$, finishing the argument for $s=1$.

For $s\geq2$, the following argument applies to every $q\geq2$. We set
\[
 M=2\lceil N_{q,s-1,m,c}\rceil
 \leq4 N_{q,s-1,m,c}.
\]
Substituting into~\eqref{eq-threshold} and using $M\leq4N_{q,s-1,m,c}$, we obtain
\begin{align*}
N_{q,1,M,c/2^{s+1}}
 &= A^{3b_q}(c/2^{s+1})^{-b_q}2^{2b_q}M\\
 &= A^{3b_q}c^{-b_q}2^{b_q(s+3)}M\\
 &\leq 4\cdot A^{3b_q}c^{-b_q}2^{b_q(s+3)}
       \cdot A^{3b_q(s-1)}c^{-b_q(s-1)}
       2^{2b_q(s-1)^2}m\\
 &= 4\cdot A^{3b_qs}c^{-b_qs}
       2^{b_q(2(s-1)^2+s+3)}m\\
 &= 4\cdot A^{3b_qs}c^{-b_qs}
       2^{2b_qs^2-b_q(3s-5)}m\\
 &= 4\cdot2^{-b_q(3s-5)}N_{q,s,m,c}\\
 &\leq N_{q,s,m,c}.
\end{align*}
For the last inequality, we used $s\geq2$ and $b_q\geq2$, which imply $b_q(3s-5)\geq2$.
The case $s=1$ therefore gives $U_1<U_2$, each of size at least $M$, with density at most $c/2^{s+1}$ in colors $1,\ldots,q-1$.

For this paragraph, call edges in the colors $1,\dots,q-1$ \emph{bad}. 
We delete from $U_1$ the vertices with more than $(c/2^s)|U_2|$ bad neighbors in $U_2$. 
Note that at most half the vertices are deleted. 
We then apply induction on $s$ inside the remaining set to obtain $2^{s-1}$ ordered sets $X_1<\cdots<X_{2^{s-1}}$ of size at least $m$, with pairwise density of bad edges at most $c$. 
For each $j$, we delete from $U_2$ the vertices having more than $c|X_j|$ bad neighbors in $X_j$. 
The degree bound on vertices of $X_j$ shows that this deletes at most $|U_2|/2^s$ vertices per $j$. 
Consequently, at least half of the vertices in $U_2$ survive. 
We apply the induction hypothesis inside the surviving set to obtain $2^{s-1}$ further ordered sets. 
Every vertex in these sets has at most $c|X_j|$ bad neighbors in each $X_j$. 
Thus the density bound holds between the two families, as well as within them. 
All recursive applications have at least $M/2\geq N_{q,s-1,m,c}$ vertices available. 
This completes the proof.
\end{proof}

We now obtain the final estimate by forcing a complete ordered multipartite graph in the last color.

\begin{thm}
\label{thm_d_degenerated_multipartite_general}
For all integers $h\geq1$ and $q\geq2$,
\[
 r_<(\underbrace{H^<,\dots,H^<}_{q-1\text{ copies}},K_h^<(\chi))
 \leq4\chi h\,A^{3b_qz}(8\chi^2h)^{b_qz}2^{2b_qz^2}.
\]
\end{thm}

\begin{proof}
We apply Lemma~\ref{lemma_multipartite_general} with the choice $s=z$, $c=(8\chi^2h)^{-1}$, and $m=4\chi h$ to a coloring on
\[
 4\chi h\,A^{3b_qz}(8\chi^2h)^{b_qz}2^{2b_qz^2}
\]
vertices. If none of the first $q-1$ colors contains an ordered copy of $H^<$, then we retain $\chi$ of the resulting ordered sets $W_1<\cdots<W_\chi$. Each has size at least $4\chi h$, and their pairwise densities in colors other than $q$ are at most $(8\chi^2h)^{-1}$. Applying Lemma~\ref{lem_conlon_final} to the graph of color-$q$ edges, with its parameter $n=h$, gives an ordered copy of $K_h^<(\chi)$ and finishes the proof.
\end{proof}

Finally, we deduce Theorem~\ref{thm_d_degenerated_multipartite_result} by a straightforward calculation.

\begin{proof}[Proof of Theorem~\ref{thm_d_degenerated_multipartite_result}]
If $\Delta=1$, then we are done by Theorem~\ref{thm_triangleFree} with $d=1$, since $21q^2\leq60\cdot2^{q-1}\leq60dp^{q-2}$ .
 We may therefore assume $\Delta\geq2$. We apply Theorem~\ref{thm_d_degenerated_multipartite_general} with $h=n$. Since $H^<$ is an ordered subgraph of $K_n^<(\chi)$, we obtain
\[
r_<(H^<;q)\leq4\chi n\,A^{3b_qz}(8\chi^2n)^{b_qz}2^{2b_qz^2}.
\]

We first estimate the base $A$. Using $d+1\leq n$, $\Delta,\chi\leq n$, and $p\geq2$, we have $ 2^{2z}<4\chi^2\leq4n^2$ and
\[
 2(d+1)(8\Delta^2\chi^2n^2)^{1/p}\leq4\sqrt2\,n^4.
\]
It follows that $A\leq4\sqrt2\,n^4+1\leq6n^4\leq n^6$. We also have $z\leq2\log{n}$, $\log_2(8\chi^2n)\leq6\log{n}$, and $\log_2(4\chi n)\leq4\log{n}$. Since $b_qz\geq 2$, taking logarithms leads to
\begin{align*}
 \log r_<(H^<;q)
 &\leq18b_qz\log{n}+b_qz\log_2(8\chi^2n)+2b_qz^2
   +\log(4\chi n)\\
 &\leq30b_qz\log{n}.
\end{align*}

The recurrence~\eqref{eq-recurrence} and $pz\geq2$ give $b_q=d\sum_{j=0}^{q-2}(pz)^j  \leq2d(pz)^{q-2}$. Consequently,
\[
 \log_2 r_<(H^<;q)
 \leq60dp^{q-2}z^{q-1}\log_2n.
\]
\end{proof}

\printbibliography[heading=bibintoc]

\end{document}